\documentclass[11pt,a4paper]{article}
\usepackage[left=3.00cm, right=3.00cm, bottom=2.00cm, top=3.00cm]{geometry}
\usepackage{amsmath, amsfonts, amsthm, amssymb, mathrsfs}
\usepackage{graphicx}
\usepackage{xcolor}
\newtheorem{teorema}{Theorem}
\newtheorem{lemma}[teorema]{Lemma}

\title{The Locating Rainbow Connection Number of the Edge Corona of a Graph with a Complete Graph}
\author{Ariestha Widyastuty Bustan$^{1,3}$\\ A.N.M. Salman$^{2}$ \\ Pritta Etriana Putri$^{2}$}
\date{}

\begin{document}

\maketitle

\begin{center}
	\textit{$^1$Doctoral Program of Mathematics, Faculty of Mathematics and Natural Sciences, Institut Teknologi Bandung, Bandung, Indonesia} \\
	\textit{$^2$Combinatorial Mathematics Research Group, Faculty of Mathematics and Natural Sciences, Institut Teknologi Bandung, Bandung, Indonesia} \\
	\textit{$^3$Mathematics Department, Faculty of Mathematics and Natural Sciences, Universitas Pasifik Morotai, Kabupaten Pulau Morotai, Indonesia} \\
\end{center}
\begin{center}
\texttt{30119004@mahasiswa.itb.ac.id} 
\end{center}

\begin{abstract}
A graph has a locating rainbow coloring if every pair of its vertices can be connected by a path passing through internal vertices with distinct colors and every vertex generates a unique rainbow code. The minimum number of colors needed for a graph to have a locating rainbow coloring is referred to as the locating rainbow connection number of a graph. Let $G$ and $H$ be two connected, simple, and undirected graphs on disjoint sets of $|V(G)|$ and $|V(H)|$ vertices, $|E(G)|$ and $|E(G)|$ edges, respectively. For $j\in\{1,2,...,|E(G_m)|\}$, the edge corona of $G_m$ and $H_n$, denoted as $G_m \diamond H_n$, is constructed by using a single copy of $G_m$ and $E(G_m)$ copies of $H_n$, and then connecting the two end vertices of the $j$-th edge of $G_m$ to every vertex in the $j$-th copy of $H_n$. In this paper, we determine the upper and lower bounds of the locating rainbow connection number for the class of graphs resulting from the edge corona of a graph with a complete graph. Furthermore, we demonstrate that these upper and lower bounds are tight.
\end{abstract}

\section{Introduction}
In \cite{Bustan-20}, Bustan et.al provided the definition of the locating rainbow connection number of a graph. Let $G$ be a simple, connected, and undirected  graph. For $k \in \mathbb{N}$, a \textit{rainbow vertex k-coloring} of $G$ is a function $c: V(G) \longrightarrow \{1,2,...,k\}$ such that for every two vertices $u$ and $v$ in $V(G)$, there exists a $u-v$ path, whose internal vertices have different colors. A path of this type is referred to as a \textit{rainbow vertex path}. To simplify the writing in this research, the rainbow vertex path will be referred to as "RVP". The smallest positive integer $k$ that allows $G$ to have a \textit{rainbow vertex k-coloring} is known as \textit{the rainbow vertex connection number} of $G$, and denoted by $rvc(G)$. Several studies related to the $rvc$ of a graph can be found in \cite{2}-\cite{5}.
 
For each $i\in \{1,2,...,k\}$, let $R_i$ be the set of vertices with color $i$ and $\Pi=\{R_1,R_2,...,R_k\}$ be an ordered partition of $V(G)$. The rainbow code of a vertex $v\in V(G)$ with respect to $\Pi$ is defined as the $k$-tuple  
$rc_{\Pi}(v)=(d(v,R_1 ),d(v,R_2 ),...,d(v,R_k ))$.
If each vertex in $G$ possesses a unique rainbow code, then $c$ is referred to as a locating rainbow k-coloring of $G$. \textit{The locating rainbow connection number} of $G$, is defined as the smallest positive integer $k$ such that $G$ has a locating rainbow k-coloring, denoted as $rvcl(G)$.

Given that every locating rainbow coloring of a graph $G$ is also a rainbow vertex coloring of $G$, then we have the following inequality:

\begin{equation}\label{eq1} 
	rvcl(G)\geq rvc(G).
\end{equation}

Several results regarding $rvcl(G)$, discovered in \cite{Bustan-20}, \cite{Bustan1}, and $rvc(G)$ in \cite{krivelevich}, are required to prove the main theorems in this paper.

\begin{lemma} \cite{krivelevich}
	Let $c$ be the number of cut  vertices in a graph $G$. Then $rvc(G)\geq c$.
	\label{cutvertex}
\end{lemma}

\begin{lemma}\cite{Bustan-20}
 Let $w,z \in V(G)$ and $w\neq z$. Let $c$ be a locating rainbow coloring of $G$. If $d(w,y)=d(z,y)$ for all $y \in V(G)-\{w,z\}$, then $c(w)\neq c(z)$.
	\label{lemma1}
\end{lemma}

\begin{teorema}\cite{Bustan1}
	Let $G$ be a connected graph of order $n \geq 3$. Then $rvcl(G)=n$ if and only if $G$ is isomorphic to complete graphs.
	\label{theoremkn}
\end{teorema}

\section{Main Results}
Let $G_m$ and $H_n$ be two graphs on disjoint sets of vertices $|V(G_m)|=m$ and $|V(H_n)|=n$, and edges $|E(G_m)|$ and $|E(H_n)|$, respectively. For $j\in\{1,2,...,|E(G_m)|\}$, the edge corona of $G_m$ and $H_n$, denoted as $G_m \diamond H_n$, is constructed by using a single copy of $G_m$ and $|E(G_m)|$ copies of $H_n$, and then connecting the two end vertices of the $j$-th edge of $G_m$ to every vertex in the $j$-th copy of $H_n$  \cite{definition}. In this research, $G_m$ is referred to as the "core" and $H_n$ as the "flare". For simplification, denote $\{n\in \mathbf{N} \mid  a \leq n\leq b\}$ by $[a,b]$. The term "\textit{entry}" denotes the distance between a vertex and a set of colors.

\subsection{Lower and Upper Bounds}

From the equation (\ref{eq1}), it can be concluded that $rvc(G)$ is one of the lower bounds for $rvcl(G)$. The same principle applies to the locating connection number of edge corona of any two graphs. Therefore, the first result we present in this subsection concerns $rvc( G_m \diamond H_n)$.

\begin{teorema}
	Let $m$ and $n$ be two integers where $m\geq2$ and $n\geq 2$. Let $G_m$ be a graph of order $m$ and $H_n$ be a graph of order $n$. Then, 
	\begin{center}
		$rvc(G_m \diamond H_n) \geq rvc(G_m)$.
	\end{center}
	\label{Lemmarvc}
\end{teorema}
\begin{proof}
	Suppose $rvc(G_m \diamond H_n) = rvc(G_m)-1$. Consequently, there are at least two vertices, $u$ and $v$, in the core that are not connected by an RVP. According to the definition of the edge corona operation, all paths connecting $u$ and $v$ in $G_m \diamond H_n$ must contain paths connecting $u$ and $v$ in the core. Therefore, there is no RVP connecting vertices $u$ and $v$ in $G_m \diamond H_n$, which leads to a contradiction. Thus, $rvc(G_m \diamond H_n) \geq rvc(G_m)$.
	
\end{proof}

Based on Lemma \ref{Lemmarvc} and Equation (\ref{eq1}), the following equation is obtained:

\begin{equation}\label{eq2} 
	rvcl(G_m\diamond H_n)\geq rvc(G_m).
\end{equation}

Additionally, besides providing general upper and lower bounds for $rvcl(G_m \diamond K_n)$, we also establish the lower bound for graph $G_m \diamond K_n$ by leveraging the completeness property of graph $K_n$. Before proceeding, we present a lemma pertaining to the $rvcl$ of a graph that contains a set of vertices equidistant to other vertices of a graph, as demonstrated in Lemma \ref{lemadjacent}.

\begin{lemma}
	Let $A$ and $B$ be two sets of vertices in graph $G$, where the vertices in each set are mutually adjacent and equidistant from the other vertices in graph $G$. If $|A|=|B|=t$, then $rvcl(G)\geq t+1$
	\label{lemadjacent}
\end{lemma}

\begin{proof}
	
	Suppose that $rvcl(G)=t$. Since $|A|=|B|=t$, based on Lemma \ref{lemma1}, we can color the vertices in sets $A$ and $B$ with $t$ different colors. Since $rvcl(G)=t$, there exists a vertex $x \in A$ colored with $k$ that is at a distance of $1$ from $R_i$ for $i\in [1,t],$ and $i\neq k$. The same applies to the vertices in set $B$. Consequently, there are $n$ pairs of vertices with the same rainbow code. Therefore, $rvcl(G)\geq t+1$.
\end{proof}

\begin{lemma}
	Let $m$ and $n$ be two integers where $m\geq2$ and $n\geq 2$. Let $G_m$ be a graph of order $m$ and $K_n$ be a complete graph of order $n$. Then, 
	\begin{center}
		$rvcl(G_m \diamond K_n) \geq n+1$.
	\end{center}
	\label{batasbawah}
\end{lemma}

\begin{proof}
	
	For $m=2$, $G_2 \diamond K_n \equiv K_{n+2}$, and according to Theorem \ref{theoremkn}, we have $rvcl(G_2 \diamond K_n)= n+2$. For $m\geq 3$, we can deduce that there are at least two flares, each containing $n$ vertices that are mutually adjacent and equidistant from other vertices in $G_m \diamond K_n$. Based on Lemma \ref{batasumum}, we obtain $rvcl(G_m \diamond K_n)\geq n+1$.
\end{proof}

\begin{teorema}
	Let $m$ and $n$ be two integers where $m\geq2$ and $n\geq 2$. Let $G_m$ be a graph of order $m$, $K_n$ be a complete graph of order $n$, and $rvc(G_m)$ be an $rvc$ of graph $G_m$. Then,
	\begin{center}
		$max\{rvc(G_m),n+1\}\leq rvcl(G_m \diamond K_n) \leq m+n+|E(G_m)|-1$.
	\end{center}
	\label{batasumum}
\end{teorema}

\begin{proof}
	The proof is divided into two stages: first, $rvcl(G_m \diamond K_n)\geq \max\{rvc(G_m),n+1\}$, and second, $rvcl(G_m \diamond K_n)\leq m+n+|E(G_m)|-1$.
	\begin{enumerate}
		\item 	From Equation (\ref{eq2}) and Lemma \ref{batasbawah}, we have $rvcl(G_m \diamond K_n)\geq max\{rvc(G_m),n+1\}$.
		\item Consider the core of $G_m \diamond K_n$ and color all vertices in the core with $m$ different colors. For the vertices in each flare, color them with distinct colors starting from $m+1, m+2, ..., m+n-1, m+n$ in the first flare, $m+2, m+3, ..., m+n, m+n+1$ in the second flare, and so on until $m+(m-1),m+(m),...,m+n+|E(G_m)|-1$ in the $|E(G_m)|$-th flare. This implies that $m+n+|E(G_m)|-1$ colors are needed to color all the vertices in $G_m \diamond K_n$.
		
		This coloring rule implies that, apart from two adjacent vertices, every pair of vertices in $G_m \diamond K_n$ is connected by a path whose internal vertices lie in the core. Since all vertices in the core are colored differently, it is evident that for every pair of vertices in $G_m \diamond K_n$, there exists an RVP. Next, we demonstrate that every vertex generates a unique rainbow code by considering the following.
		\begin{enumerate}
			\item Each vertex in the core has a distinct color, ensuring that every vertex in the core has a unique rainbow code.
			\item The colors $1, 2, 3, ..., m$ are exclusively assigned for the vertices in the core, ensuring that no vertex in the core shares the same rainbow code with vertices in the flare.
			\item Each vertex within a single flare is colored differently, thereby resulting in unique rainbow codes.
			\item If any vertices between flares have the same color, note that each flare contains vertices colored with a different set of colors. Therefore, the rainbow code for each vertex between the flares is different.
		\end{enumerate}
		Based on these four conditions, we can conclude that every vertex in $G_m \diamond K_n$ generates a unique rainbow code. Therefore, we have $rvcl(G_m \diamond K_n)\leq m+n+|E(G_m)|-1$.
	\end{enumerate}
	
\end{proof}

\begin{figure}[h]
	\centering
	\caption{A locating rainbow coloring of $G_m \diamond K_n$}
	\includegraphics[width=5in]{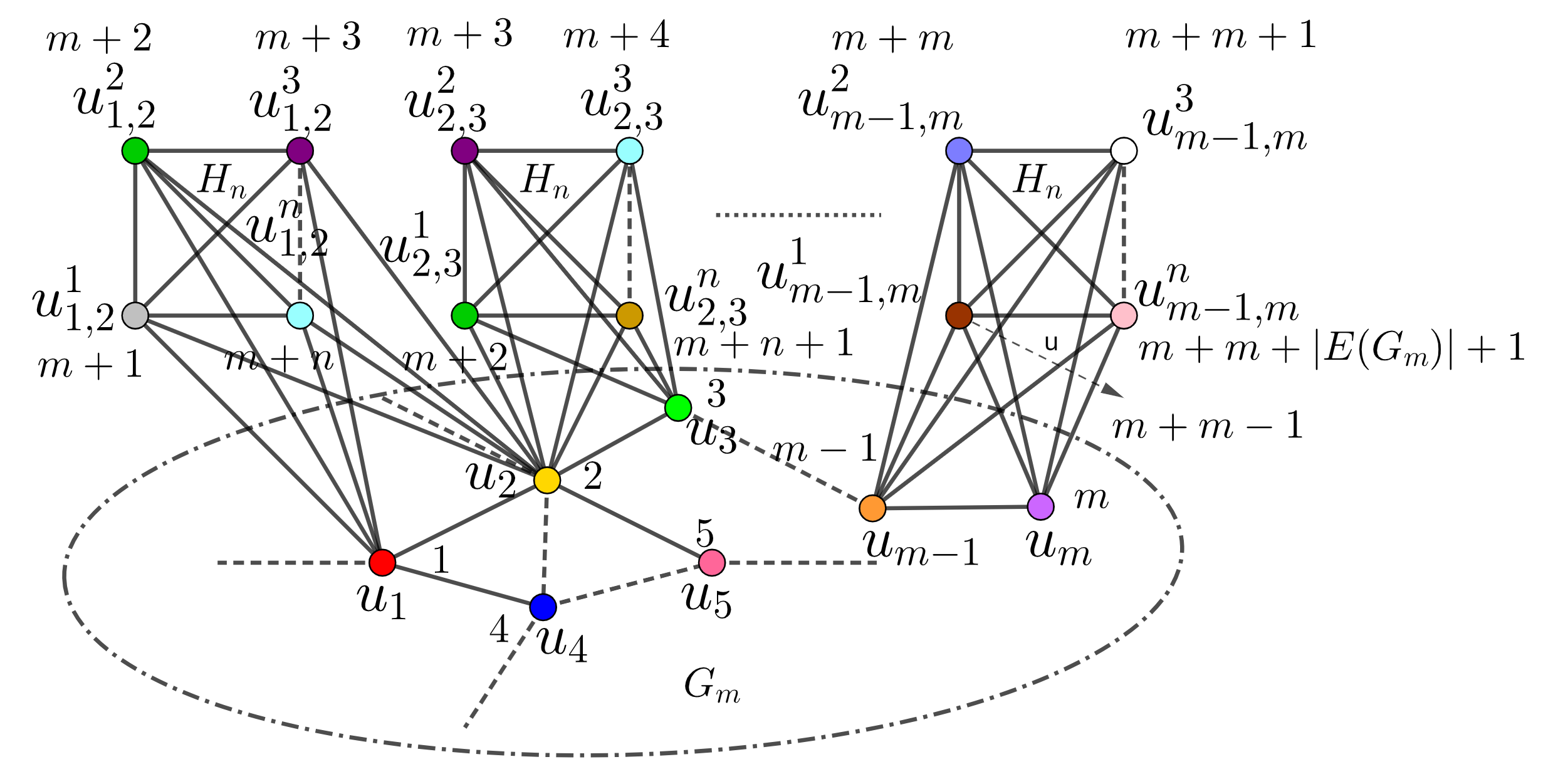}
	\label{batasataskoronasisi}
\end{figure}

Figure \ref{batasataskoronasisi} illustrates a locating rainbow coloring on the graph $G_m \diamond K_n$ using $m+n+|E(G_m)|-1$ colors. In the following subsection, we demonstrate the tightness of the upper and lower bounds in Theorem \ref{batasumum} by providing the $rvcl$ of the edge corona of several classes of graphs. These graphs include trees that have the same $rvcl$ as the upper and lower bounds in Theorem \ref{batasumum}, as well as graphs whose $rvcl$ lies between those upper and lower bounds.

\subsection{Locating Rainbow Connection Number of Edge Corona of Trees with Complete Graphs}

In this subsection, we first determine the $rvc$ of the edge corona of a tree with a complete graph. This number will serve as a strict lower bound for the $rvcl$ of the edge corona of the tree with a complete graph.

On the other hand, Theorem \ref{Lemmarvc} demonstrates that $rvcl(G_m \diamond K_n) \geq rvc(G_m)$. Let $T_m$ be a tree of order $m$. Furthermore, Theorem \ref{teotree} shows that if $G_m$ is isomorphic to $T_m$, we can obtain the exact value of the lower bound as indicated in Theorem \ref{Lemmarvc}.

\begin{teorema}
	Let $m$ and $n$ be two integers where $m\geq2$ and $n\geq 2$. Let $T_m$ be a graph of order $m$ and $H_n$ be a graph of order $n$. Then, 
	\begin{center}
		$rvc(T_m \diamond H_n) = rvc(T_m)$.
	\end{center}
	\label{teotree}
\end{teorema}
\begin{proof}

	Based on Theorem \ref{Lemmarvc}, we have $rvc(T_m \diamond H_n) \geq rvc(T_m)$. Furthermore, any two non-adjacent vertices in $T_m \diamond H_n$ are always connected by a path whose internal vertices are cut vertices. Therefore, we can establish $rvc(T_m \diamond H_n) \leq rvc(T_m)$ by coloring all cut vertices with unique colors and other vertices with color $1$. Thus, $rvc(T_m \diamond H_n) = rvc(T_m)$.
\end{proof}

Figure \ref{rvcTm} illustrates a rainbow vertex coloring on a tree graph with seven vertices and the edge corona graph of a tree of order seven with a complete graph of order three. It can be seen that both graphs have the same rainbow vertex connection number.

\begin{figure}[h]
	\centering
		\caption{A rainbow vertex coloring of (a) $T_7$ and (b) $T_7 \diamond K_3$}
	\includegraphics[width=4.1in]{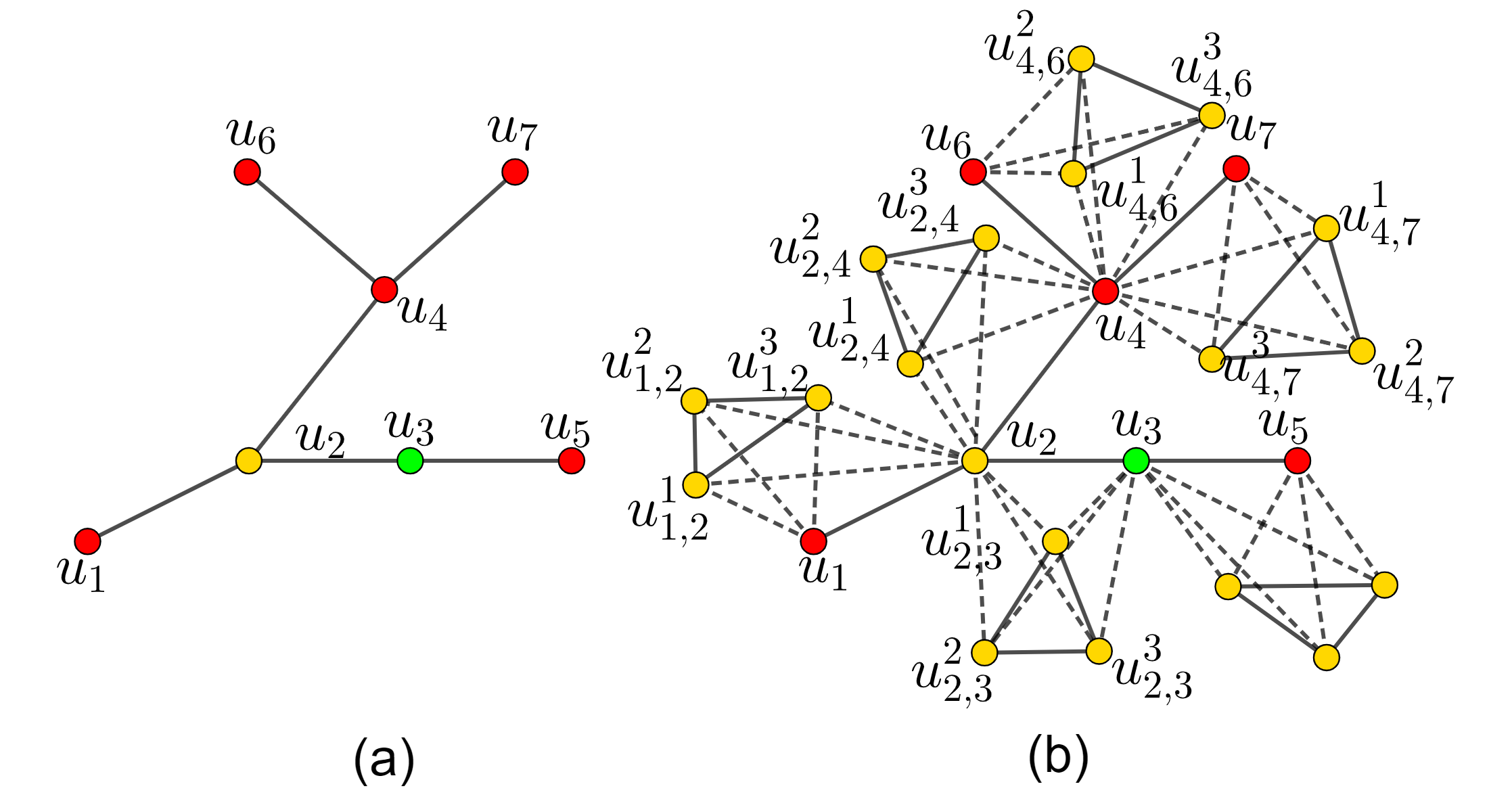}
	\label{rvcTm}
\end{figure}

\newpage

Graph $T_m$ is a tree of order $m$ with $m-1$ edges. According to Theorem \ref{batasumum}, the maximum value of $rvcl(T_m \diamond K_n)$ is $m+n+|E|-1$ or equivalently $m+n+m-2=2m+n-2$.
\\

Every tree always has at least two pendant vertices. If $m=2$, which corresponds to a single edge, $T_2 \diamond K_n \equiv K_{n+2}$. According to Theorem \ref{theoremkn}, it clearly requires $n+2$ colors. If $m\geq 3$, then there are at least two disjoint pendant edges. Each end vertex of these edges is connected to all vertices in the flare in $T_m \diamond K_n$, creating at least two subgraphs that are complete graphs with $n+1$ vertices. 

Based on Lemma \ref{batasbawah}, at least $n+2$ colors are required so that every vertex generates a unique rainbow code. Therefore, the upper and lower bounds of $rvcl(T_m \diamond K_n)$ are as follows.

\begin{lemma}
	Let $m$ and $n$ be two integers where $m\geq2$ and $n\geq 2$. Let $T_m$ be a tree of order $m$, $K_n$ be a complete graph of order $n$, and $rvc(T_m)$ be an $rvc$ of graph $T_m$. Then 
	\begin{center}
		$max\{rvc(T_m),n+2\}\leq rvcl(T_m \diamond K_n) \leq 2m+n-2$. 
		\label{lemmabatasumumpohon}
	\end{center}
\end{lemma}
$P_2 \diamond K_n \equiv K_{n+2}$, which has a $rvcl$ equal to the upper bound of Lemma \ref{batasumum} and lemma \ref{lemmabatasumumpohon}. Furthermore, in the theorem below, we will demonstrate that the $rvcl$ of a path with a complete graph is equal to the lower bound of Lemma \ref{lemmabatasumumpohon}.
For ease of color labelling rules, we sequentially label each vertex and edge in the resulting edge corona of a path of order $m$ and a complete graph with order $n$ sequentially, as follows. Let $V(P_m\diamond K_n)=\{u_i|i\in[1,m]\} \cup \{u_{i,i+1}^j|i\in[1,m-1], j\in [1,n]\}$, and $E(P_m\diamond K_n)=\{u_iu_{i+1}|i\in[1,m-1]\}\cup  \{u_iu_{i,i+1}^j|i\in[1,m-1], j\in[1,n]\}  \cup \{u_{i,k}^ju_{i,l}^j|i,j\in[1,m], k\neq l, j\in[1,n] \}$. $diam(P_m\diamond K_n)=m-1$. Refer to Figure \ref{pmkn} for an illustration of the graph $P_m \diamond K_n$.
\begin{figure}[h]
	\centering
		\caption{Graph $P_m \diamond K_n$}
	\includegraphics[width=3.4in]{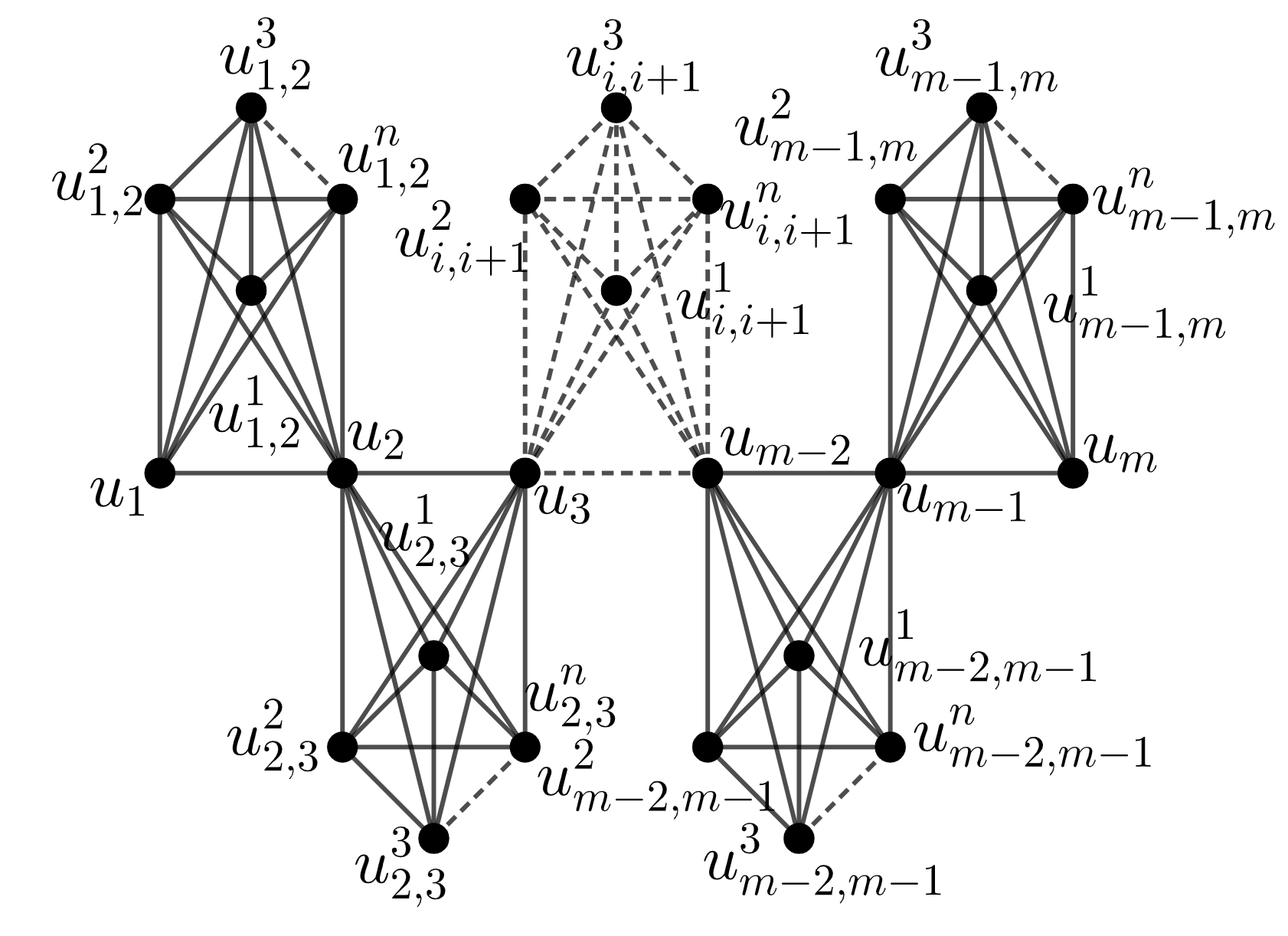}

	\label{pmkn}
\end{figure}

\begin{teorema}
	Let $m$ and $n$ be two integers where $m\geq2$ and $n\geq 2$. Let $P_m$ be a path of order $m$ and $K_n$ be a complete graph of order $n$. Then 
	\begin{center}
		$rvcl(P_m \diamond K_n) = max\{rvc(P_m),n+2\}$.
	\end{center}
	\label{Theorempath}
\end{teorema}

\begin{proof}
	The proof is divided into two parts: $rvcl(P_m \diamond K_n)\geq max\{rvc(P_m),n+2\}$ and $rvcl(P_m \diamond K_n)\leq max\{rvc(P_m),n+2\}$.
	\begin{enumerate}
		\item Based on (\ref{eq2}), we have $rvcl(P_m \diamond K_n)\geq rvc(P_m)$. Then, using the same method as deriving the lower bound from Lemma \ref{lemmabatasumumpohon}, we also get $rvcl(P_m \diamond K_n)\geq n+2$.
		\item To show that $rvcl(P_m \diamond K_n)\leq max\{rvc(P_m),n+2\}$, we construct a vertex coloring $c:V(P_m \diamond K_n)\longrightarrow[1, max\{rvc(P_m),n+2\}]$ as follows:
		\begin{center}
			$c(u_1)=1$.\\
			$c(u_m)=m-2.$\\
			$c(u_i)=i-1.$ for $i\in[2,m-1]$.
		\end{center}
		
		\begin{center}
			$\begin{array}{ccl}
				c(u_{i,i+1}^j)&=&\left\{
				\begin{array}{ll}
					j+1, & \hbox{for $i=1$ ;}\\
					(i+j-2)~mod~(max\{rvc(P_m),n+2\})+1, & \hbox{for $i\in[2,m-1]$.}
				\end{array}
				\right.
			\end{array}$
		\end{center}

		Based on these coloring rules, it is determined that, except for two adjacent vertices, every pair of vertices in $P_m \diamond K_n$ is connected by a path whose internal vertices belong to the core. Since all vertices in the core have distinct colors, it is evident that for every two vertices in $P_m \diamond K_n$, there exists an RVP. Next, we demonstrate that every vertex generates a unique rainbow code by considering the following conditions.
		\begin{enumerate}
			\item $c(u_i)\neq c(u_j)$ for distinct $i,j\in[2,m-1]$, therefore $rc_{\Pi}(u_i)\neq rc_{\Pi}(u_j)$.
			\item 
			$c(u_1)=c(u_2)=1$, but $u_1$ and $u_2$ share neighbors with vertices in one of the flare that do not contain the color $max\{rvc(P_m),n+2\}$. Therefore, there must be at least one vertex $z$ such that $c(z)=max\{rvc(P_m),n+2\}$ and $d(u_1,z)>d(u_2,z)$, resulting in $rc_{\Pi}(u_1)\neq rc_{\Pi}(u_2)$.
			
			\item $c(u_{m-1})=c(u_{m})=m-2$, but $u_{m-1}$ and $u_{m}$ have common neighbors with vertices in one of the flare that do not have the color $z$. Therefore, there must be at least one vertex $y$ such that $c(y)=z$ and $d(u_{m-1},y)<d(u_m,y)$, resulting in $rc_{\Pi}(u_{m-1})\neq rc_{\Pi}(u_m)$.
			
			\item $c(u_{i,j})\neq c(u_{i,k})$ for $i\in[1,m]$, and distinct $j,l\in[1,n]$. 
			
			\item  The vertices in each flare are not colored with the same set of colors; hence, the vertices in different flares have different rainbow codes.
			
		\end{enumerate}
		Based on these five conditions, it can be concluded that every vertex in $P_m \diamond K_n$ generates a unique rainbow code. Therefore, we obtain $rvcl(P_m \diamond K_n)=max\{rvc(P_m),n+2\}$.
	\end{enumerate}
\end{proof}

Theorem \ref{Theorempath} illustrates that the path length or diameter significantly influences the $rvcl$ of the edge corona of a path with a complete graph. Conversely, in trees with small diameters, such as stars, it is the order of the complete graph that affects the $rvcl$ of the edge corona. Related results have been demonstrated in \cite{bustan2023locating}.

\subsection{Locating Rainbow Connection Number of Edge Corona of Cycles with Complete Graphs}

Furthermore, in \cite{fauziah}, the $rvc$ of the edge corona between a cycle and a path, as well as between two cycles, has been determined. In this paper, we establish the $rvc$ of the edge corona between a cycle graph with any graph $H_n$, as presented in Theorem \ref{rvccycle}.

Let $V(C_m\diamond K_n)=\{u_i|i\in[1,m]\} \cup \{u_{i,i+1}^k| i\in [1,m-1], k\in[1,n]\} \cup \{u_{m,1}^k|k\in[1,n]\}$. $E(C_m\diamond K_n)=\{u_iu_j|i,j\in[1,m-1], j=i+1\} \cup \{u_mu_1\}$ $\cup \{u_iu_{i,j}^k|i,j\in[1,m], k\in[1,n]\} \cup \{u_iu_{i-1,i}^k|i\in[2,m], k\in[1,n]\} \cup \{u_1u_{m-1,1}^k|k\in[1,n]\} \cup \{u_{i,i+1}^ku_{i,i+1}^l|i,j\in [1,m-1], k,l\in[1,n], k\neq l\} \cup \{u_{m,1}^ku_{m,1}^l k,l\in[1,n], k\neq l\}$. Figure \ref{cmkn} illustrates the edge corona of a cycle graph with a complete graph.

\begin{figure}[h]
	\centering
	\includegraphics[width=3.3in]{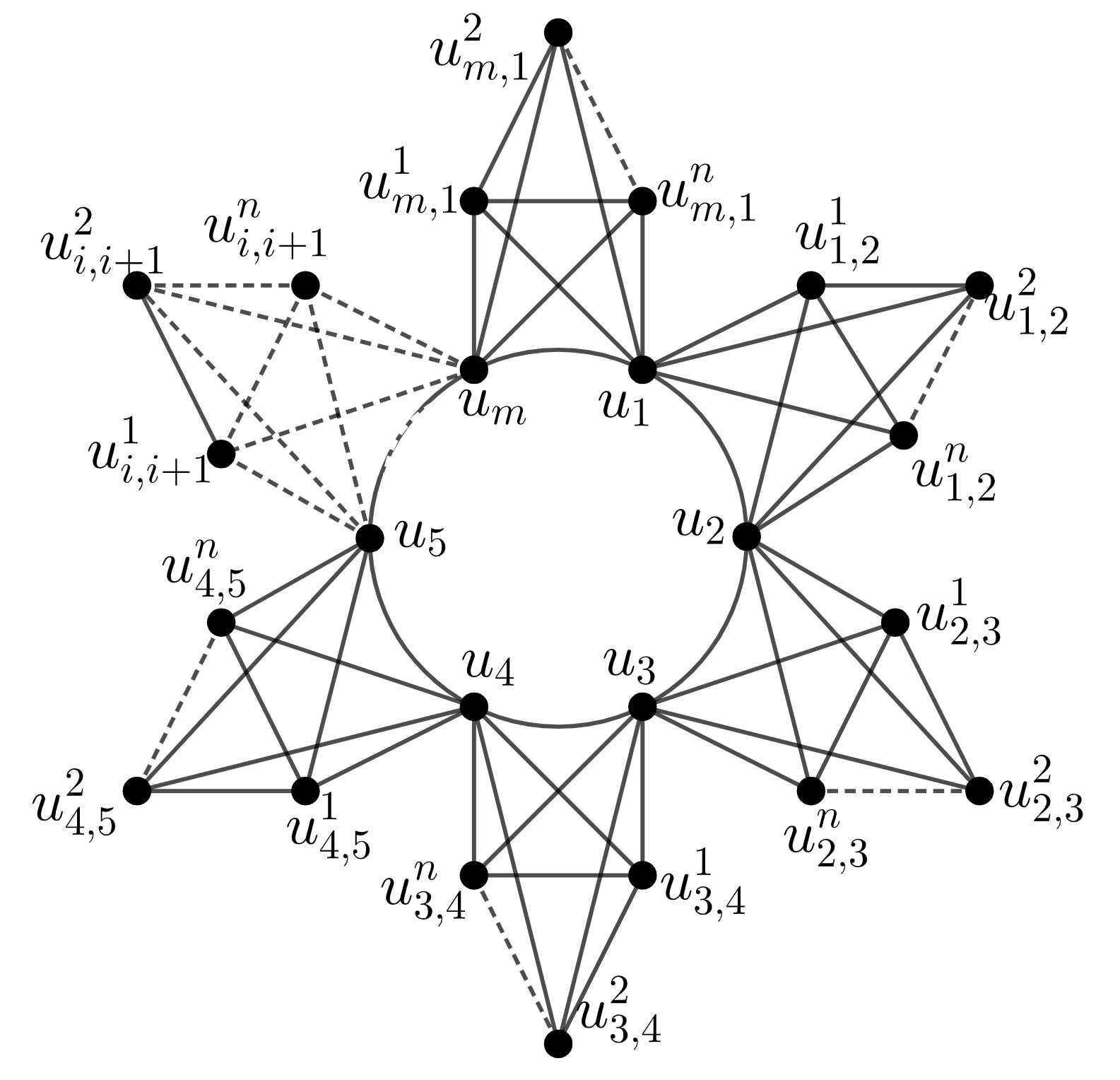}
	\caption{Graph $C_m \diamond K_n$}
	\label{cmkn}
\end{figure}

\begin{teorema}
	Let $m$ and $n$ be two integers where $m\geq2$ and $n\geq 2$. Let $C_m$ be a graph of order $m$ and $H_n$ be a graph of order $n$. Then, 
	\begin{center}
		$\begin{array}{ccl}
			rvc(C_{m} \diamond H_n)&=&\left\{
			\begin{array}{ll}
				1, & \hbox{for $m =3$ ;} \\
				\lceil\frac{m}{2}\rceil, & \hbox{for $m\geq4$.}
			\end{array}
			\right.
		\end{array}$
	\end{center}
	\label{rvccycle}
\end{teorema}
\begin{proof}
	
	\begin{enumerate}
		\item $m=3$\\
		Since $diam(C_3 \diamond H_n)=2$, it is clear that $rvc(C_3 \diamond H_n)=1$ by coloring all vertices with the same color.
		\item $m$ even \\
		Since $diam(C_m\diamond H_n)=\lceil\frac{m}{2}\rceil +1$ for even $m$, we have $rvc(C_m\diamond H_n)\geq \lceil\frac{m}{2}\rceil$. To show that $rvc(C_m\diamond H_n)\leq \lceil\frac{m}{2}\rceil$, we assign different colors to the vertices in the core within distance $\lceil\frac{m}{2}\rceil$, and color the vertices in the flare with color $1$. Therefore, any two vertices in $C_m \diamond H_n$ are always connected by an RVP. Thus, $rvc(C_m\diamond H_n)= \lceil\frac{m}{2}\rceil$.
		\item $m$ odd\\
		Suppose $rvc(C_m \diamond H_n)=\lceil\frac{m}{2}\rceil-1$. This implies that there exists at least one path with a maximum length of $\lfloor\frac{m+1}{4}\rfloor$ containing at least two vertices of the same color, which prevents $C_m \diamond H_n$ from having a rainbow vertex coloring. Therefore, $rvc(C_m \diamond H_n)\geq \lceil\frac{m}{2}\rceil$. By applying the same coloring rules as for the case of an even $m$, we get $rvc(C_m \diamond H_n)=\lceil\frac{m}{2}\rceil$.
	\end{enumerate}
	
\end{proof}

\begin{teorema}
	Let $m$ and $n$ be two integers where $m\geq3$ and $n\geq 2$. Let $C_m$ be a cycle of order $m$, and $K_n$ be a complete graph of order $n$. Then, 
	
	\begin{center}
		$\begin{array}{ccl}
			rvcl(C_{m} \diamond K_n)&=&\left\{
			\begin{array}{ll}
				n+1, & \hbox{for $m=3$ or $m\geq 4$ and $n\geq m-1$ ;}\\
				max\{\lceil\frac{m}{2}\rceil+1, n+2, \}, & \hbox{ $m\geq 4$ and $\lceil\frac{m}{2}\rceil -1\leq n<m-1 $;}\\
				\lceil\frac{m}{2}\rceil, & \hbox{for $m\geq5$ and $n\leq \lceil\frac{m}{2}\rceil-2$.}
			\end{array}
			\right.
		\end{array}$
	\end{center}
\end{teorema}

\begin{proof}
	
	\begin{enumerate}
		\item $m=3$ or $m\geq 4$ and $n\geq m-1$. \\
		Based on Lemma \ref{lemma1}, we have $rvcl(C_m\diamond K_n)\geq n+1$. Next, we will demonstrate $rvcl(C_m\diamond K_n)\leq n+1$ by constructing the following vertex coloring. $c:(V(C_m\diamond K_n)) \longrightarrow[1, n+1]$.
		
		\begin{center}
			$c(u_i)=i$ for $i\in[1,m]$.\\
				$c(u_{i,j}^k)=((i+(k-2))~mod~(n+1))+1$ for $i\in[1,m]$, $k\in[1,n]$, $j=i+1$ for $i\neq m$, $j=1$ for $i=m$. 
		\end{center}
		
		Next, consider several conditions regarding the rainbow codes of the vertices in $C_m\diamond K_n$ after applying the coloring.
		\begin{enumerate}
			\item Each flare has a set of colors, not all of which are the same.
			\item Each vertex in the core is colored differently. 
			\item Each vertex in the core is adjacent to the vertices of two flares, each colored with $n$ different colors. Since each flare is colored with a set of colors where not all elements are the same, at least each vertex in the core is at a distance of $1$ from the vertices colored with $n$ different colors. Meanwhile, each vertex in the flare is at a distance of $1$ from the vertices colored with $n-1$ different colors.
		\end{enumerate}
		Thus every vertex generates a unique rainbow code. Therefore, we get $rvcl(C_m\diamond K_n) \leq   n +1$. Hence, it is proven that $rvcl(C_m \diamond K_n) = n + 1$.
		
		\item $m\geq 4$ and $\lceil\frac{m}{2}\rceil -1\leq n<m-1 $.\\
		\begin{enumerate}
			\item $n=\lceil\frac{m}{2}\rceil-1$.\\
			Suppose $rvcl(C_m\diamond K_n)= \lceil\frac{m}{2}\rceil$. Since $n=\lceil\frac{m}{2}\rceil-1$, there are at least two flares with a color set where all elements are identical. Without a loss of generality, assume one of the flares contains the vertices $u_{1,2}^k$ for $k\in[1,n]$. According to Lemma \ref{lemma1}, $u_{1,2}^k$ are colored with $\lceil\frac{m}{2}\rceil-1$ different colors. This means that only one color out of $\lceil\frac{m}{2}\rceil$ is not used to color the vertices $u_{1,2}^k$. Without a loss of generality, assume that color be $\lceil\frac{m}{2}\rceil$, and $c(u_{1,2}^k)=k$ for $k\in[1, \lceil\frac{m}{2}\rceil-1]$ .
			
			Consider the case when $m=5$. There are two possible colorings for the vertices $u_1$ and $u_2$: either $c(u_1)\textcolor{white}{i}=c(u_2)$ or $c(u_1)\neq\textcolor{white}{i}c(u_2)$.
			
			If $c(u_1)=\textcolor{white}{i}c(u_2)=x$, then there must be at least three vertices with the same color and at a distance of one from an $R_i$ for\textcolor{white}{i} $i\in[1,\lceil\frac{m}{2}\rceil-1]$ and\textcolor{white}{i} $i\neq x$. As a result, these three vertices must have different distances to $R_{\lceil\frac{m}{2}\rceil}$. Since $diam(C_5 \diamond K_{n})= 3$, the only possible entry with respect to $R_{\lceil\frac{m}{2}\rceil}$ are $1$, $2$, and $3$.
			
			On the other hand, if one of them has a distance of $3$ to $R_{\lceil\frac{m}{2}\rceil}$, then no other vertex can have an entry of $1$ with respect to $R_{\lceil\frac{m}{2}\rceil}$. This leads to at least two vertices having the same rainbow code, which is a contradiction.
			
			Next, for the case of $c(u_1)\neq c(u_2)$, we combine the proof with the case of $m\geq 6$, where for $m\geq 6$  every two vertices are connected by an RVP if $c(u_1)\neq c(u_2)$.
			
			Since the set of vertices $\{u_{1,2}^k|k\in[1,n]\}\cup \{u_1, u_2\}$ forms a complete graph with order $\lceil\frac{m}{2}\rceil +1$, by Lemma \ref{lemma1}, $c(u_1)=c(u_2)\neq \lceil\frac{m}{2}\rceil$. Therefore, without a loss of generality, let $c(u_1)=1$ and $c(u_2)=2$.
			
			Next, there are four possible distances between the vertices $u_1$ and $u_2$ with respect to $R_{\lceil\frac{m}{2}\rceil}$. \begin{enumerate}
				\item 	If $d(u_1, R_{\lceil\frac{m}{2}\rceil})>d(u_2,R_{\lceil\frac{m}{2}\rceil})$, then $rc_{\Pi}(u_1)=rc_{\Pi}(u_{1,2}^1)$. 
				\item If $d(u_1, R_{\lceil\frac{m}{2}\rceil})<d(u_2,R_{\lceil\frac{m}{2}\rceil})$, then $rc_{\Pi}(u_2)=rc_{\Pi}(u_{1,2}^2)$.
				\item  
				If $d(u_1, R_{\lceil\frac{m}{2}\rceil})=d(u_2,R_{\lceil\frac{m}{2}\rceil})\neq 1$, this means that the color $\lceil\frac{m}{2}\rceil$ is used in a flare that is not adjacent to $u_1$ and $u_2$. Consequently, there are two other flares that have a set of colors $\{1,2,...,\lceil\frac{m}{2}\rceil-1\}$ and are at a distance of $2$ from $R_{\lceil\frac{m}{2}\rceil}$.As a result, the vertices in both of these flares have the same rainbow code.
				\item If $d(u_1, R_{\lceil\frac{m}{2}\rceil})=d(u_2,R_{\lceil\frac{m}{2}\rceil})= 1$. Consequently, the color $\lceil\frac{m}{2}\rceil$ must be assigned to one vertex in a flare containing the vertices $u_{m,1}^k$ and one vertex in a flare containing the vertices $u_{2,3}^k$ for $k\in[1,n]$. Next, consider another flare where the vertices are colored with $[1, \lceil\frac{m}{2}\rceil-1]$. Repeat possibilities (a), (b), (c), until possibility (d) occurs again. Then, there must be at least $n$ pairs of vertices with the same rainbow code.
				
			\end{enumerate}
			Based on (i), (ii), (iii), and (iv), we get a contradiction. Thus $rvcl(C_m\diamond K_n) \geq  \lceil\frac{m}{2}\rceil +1$. For $m\leq 6$, we have $rvcl(C_m\diamond K_n)\leq \lceil\frac{m}{2}\rceil +1$ as shown in Figure \ref{c5}. 
			
			\begin{figure}
				\caption{A locating rainbow coloring of (a) $C_5 \diamond K_2$ and (b) $C_6 \diamond K_2$.}
				\includegraphics[width=0.85\textwidth]{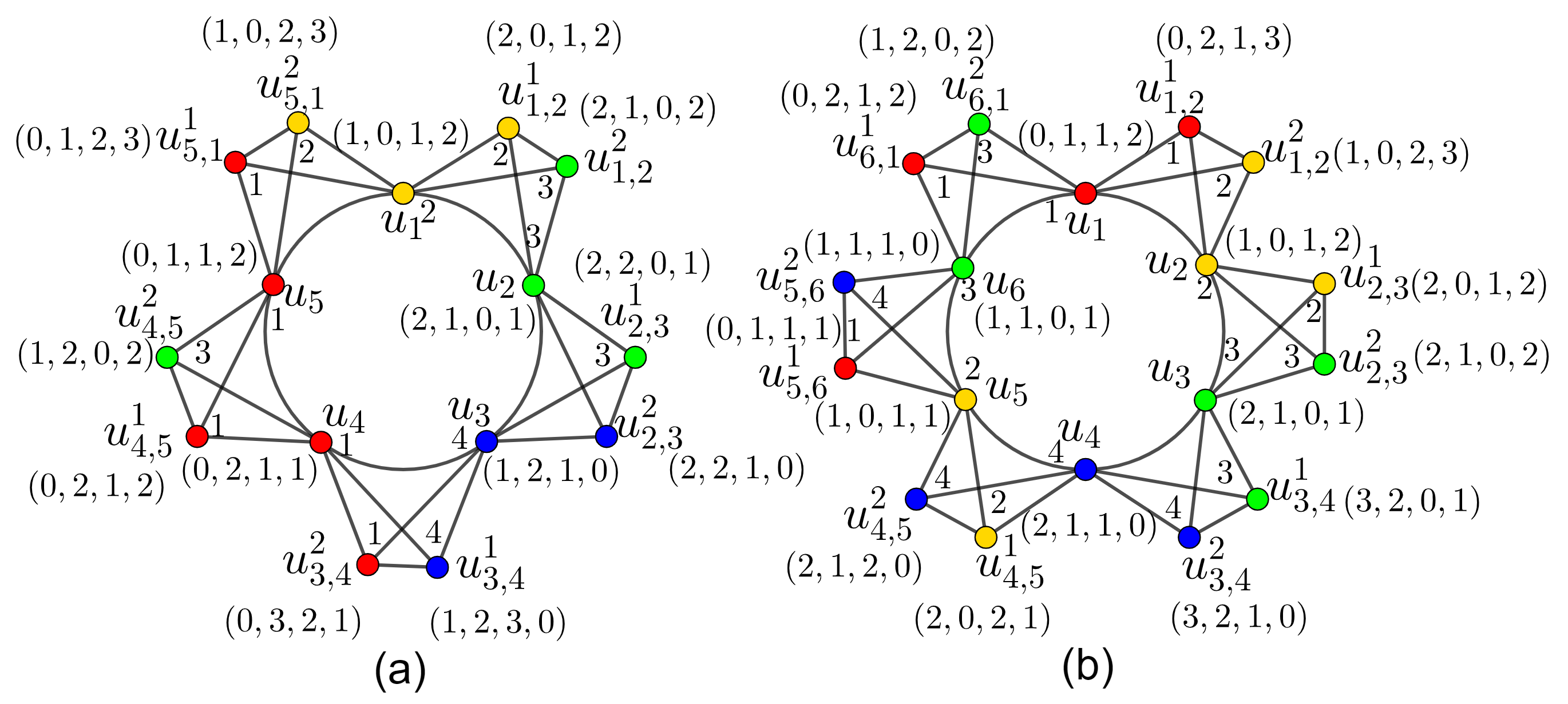}
				\label{c5}
			\end{figure}
			
			Next, we demonstrate $rvcl(C_m\diamond K_n)\leq \lceil\frac{m}{2}\rceil+1$ for $m\geq 7$ by constructing the following vertex coloring. $c(V(C_m\diamond K_n)): \longrightarrow[1, \lceil\frac{m}{2}\rceil+1]$.
			
			\begin{center}
				$c(u_i)=	((i-1)~mod~ (\lceil\frac{m}{2}\rceil))+1$ for $i\in[1,m]$.
			\end{center}
			\begin{center}				$\begin{array}{ccl}
					c(u_{i,j}^k)&=&\left\{
					\begin{array}{ll}
						((i+(k-2))~mod~ (\lceil\frac{m}{2}\rceil))+1, & \hbox{for $i\in [1, \lceil\frac{m}{2}\rceil -2]\cup$}\\
						~ & \hbox{ $ [\lceil\frac{m}{2}\rceil,m]$,} \\
						~ & \hbox{$j=i+1$ for $i\neq m$, } \\
						~ & \hbox{$j=1$ for $i=m$,}\\
						~ & \hbox{$k\in[1,n]$, $(i+k-2)<m$;}\\
						((i+(k-2))~mod~ (\lceil\frac{m}{2}\rceil +1))+1, & \hbox{for $i=\lceil\frac{m}{2}\rceil-1$, $k\in[1,n]$, }\\
						~ & \hbox{$i+(k-2)\leq m$, $j=i+1$;}\\
						((i+(k-2))~mod~ (m))+1, & \hbox{for $i\in [\lfloor\frac{m}{2}\rfloor +1,m]$,}\\
						~& \hbox{$j=i+1$ for $i\neq m$,}\\
						~& \hbox{$j=1$ for $i=m$, }\\
						~& \hbox{$k\in[1,n]$, $(i+k-2)\geq m$ .}
						
					\end{array}
					\right.
				\end{array}$
			\end{center}
			
			From these coloring rules, it is obtained that any two vertices in the core, are always given different colors at a maximum distance of $\lfloor\frac{m}{2}\rfloor$. This ensures that for every two vertices in  $C_m \diamond K_n$, there exists an RVP.
			
			Next, consider several conditions regarding the color codes of the vertices in $C_m\diamond K_n$ after applying the coloring.
			\begin{enumerate}
				\item At most, two flares have a set of colors where all elements are the same, but they have different distances to the color $R_{\lceil\frac{m}{2}\rceil+1}$.
				\item One color is used at most twice in the core.
				\item When $c(u_i)=c(u_j)$ for distinct $i,j\in [1,m]$, $d(u_i,R_{\lceil\frac{m}{2}\rceil+1})\neq d(u_j, R_{\lceil\frac{m}{2}\rceil+1})$ follows.
				\item 
				Each vertex within the core is adjacent to vertices in two flares, each colored with $n$ different colors. Since each flare is colored with a set of colors where not all elements are the same, every vertex in the core is at least at a distance of $1$ from vertices colored with $n$ different colors. Additionally, it is known that every vertex in a flare is at a distance of $1$ from vertices colored with $n-1$ different colors.
			\end{enumerate}
			Therefore, we get $rvcl(C_m\diamond K_n) =  \lceil\frac{m}{2}\rceil +1$. 
			
			\item $n+2> \lceil\frac{m}{2}\rceil +1$.
			Suppose $rvcl(C_m\diamond K_n) = n +1$. Since $m>n+1$, there must be at least two flares whose vertices are colored with the same set of colors. Using the same method as in subcase (a), a contradiction is obtained. Thus, we must have $rvcl(C_m\diamond K_n) \geq n +2$.
			
			Based on Figure \ref{c6} and Figure \ref{c7}, we have $rvcl(C_m\diamond K_n)\leq n+2$ for $m\leq 6$. Next, for other values of $m$, we demonstrate $rvcl(C_m\diamond K_n)\leq n+2$ by constructing the following vertex coloring. Let $c:(V(C_m\diamond K_n)) \longrightarrow[1, n+2]$.
			
			\begin{figure}
				\caption{A locating rainbow coloring of (a) $C_4 \diamond K_2$ and (b) $C_5 \diamond K_3$.}
				\includegraphics[width=1\textwidth]{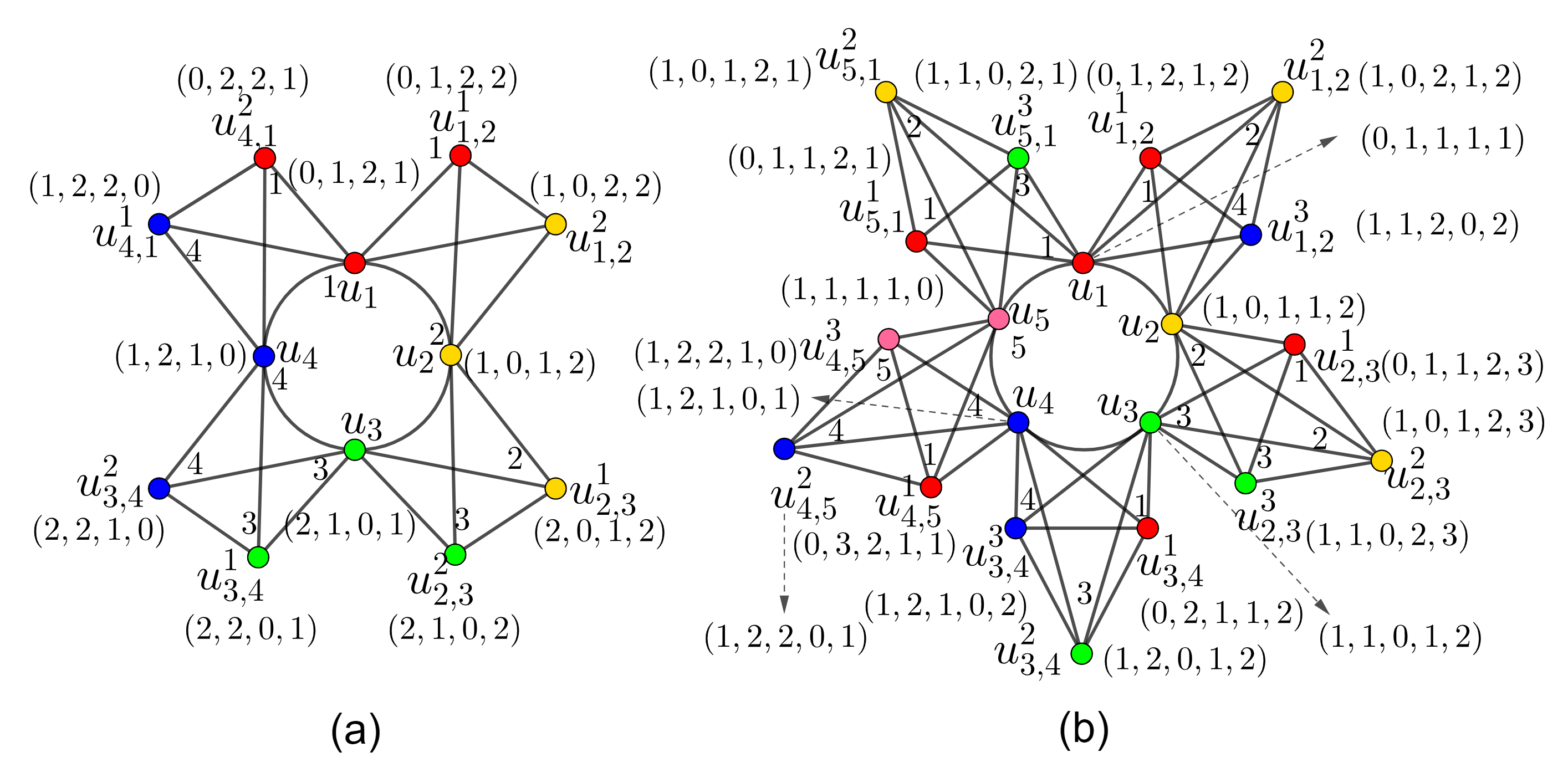}
				\label{c6}
			\end{figure}
		
			\begin{figure}
			\caption{A locating rainbow coloring of (a) $C_6 \diamond K_3$ and (b) $C_6 \diamond K_4$.}
			\includegraphics[width=1\textwidth]{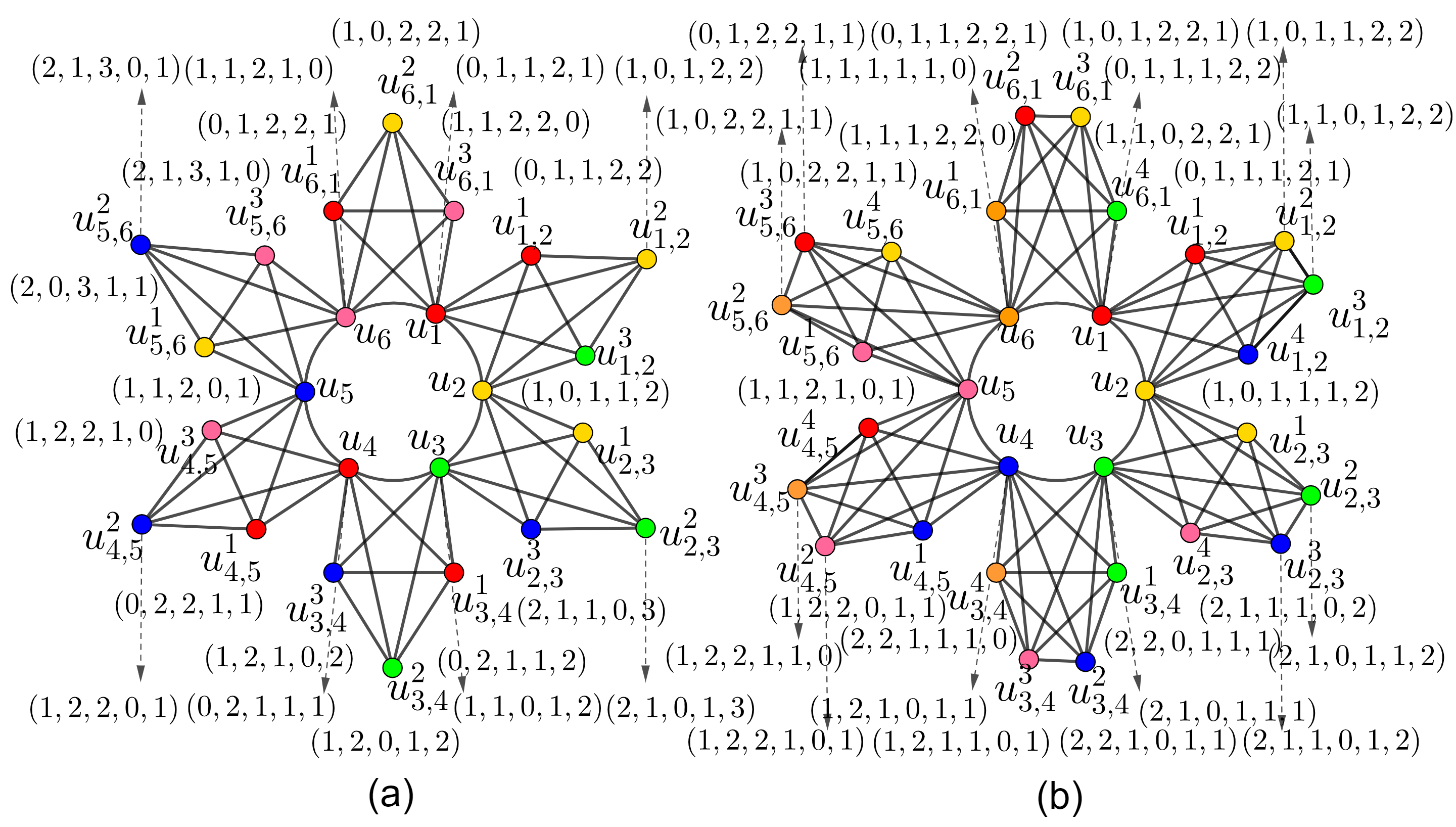}
			\label{c7}
		\end{figure}
			
			\begin{center}
				$c(u_i)=	((i-1)~mod~ (\lceil\frac{m}{2}\rceil))+1$ for $i\in[1,m]$.
			\end{center}
			\begin{center}
				$\begin{array}{ccl}
					c(u_{i,j}^k)&=&\left\{
					\begin{array}{ll}
						((i+(k-2))~mod~ (n+1))+1, & \hbox{for $i\in [1, \lceil\frac{m}{2}\rceil -2]\cup$}\\
						~ & \hbox{ $ [\lceil\frac{m}{2}\rceil,m]$, } \\
						~ & \hbox{$j=i+1$ for $i\neq m$, } \\
						~ & \hbox{$j=1$ for $i=m$,}\\
						~ & \hbox{$k\in[1,n]$, $(i+k-2)<m$;}\\
						((i+(k-2))~mod~ (n+2))+1, & \hbox{for $i=\lceil\frac{m}{2}\rceil-1$, $k\in[1,n]$, }\\
						~ & \hbox{$i+(k-2)\leq m$, $j=i+1$;}\\
						((i+(k-2))~mod~ (m))+1, & \hbox{for $i\in [\lfloor\frac{m}{2}\rfloor +1,m]$,}\\
						~& \hbox{$j=i+1$ for $i\neq m$,}\\
						~& \hbox{$j=1$ for $i=m$, }\\
						~& \hbox{$k\in[1,n]$, $(i+k-2)\geq m$ .}
						
					\end{array}
					\right.
				\end{array}$
			\end{center}
			
			From these coloring rules, it is obtained that for any two vertices in the core, are always given different colors at a maximum distance of $\lfloor\frac{m}{2}\rfloor$. This ensures that for every two vertices in $C_m \diamond K_n$, there exists an RVP.
			
			Next, consider several conditions regarding the color codes of the vertices in $C_m\diamond K_n$ after applying the coloring.
			\begin{enumerate}
				\item 
				At most, two flares have a set of colors where all elements are the same, but they have different distances to the color $R_{n+2}$.
				\item 
				One color is used at most twice in the core.
				\item When $c(u_i)=c(u_j)$ for distinct $i,j\in [1,m]$, $d(u_i,R_{n+2})\neq d(u_j, R_{n+2})$ follows.
				\item Each vertex within the core is adjacent to vertices in two flares, each colored with $n$ different colors. Since each flare is colored with a set of colors where not all elements are the same, every vertex in the core is at least at a distance of $1$ from vertices colored with $n$ different colors. Additionally, it is known that every vertex in a flare is at a distance of $1$ from vertices colored with $n-1$ different colors.
			\end{enumerate}
			Therefore, we get $rvcl(C_m\diamond K_n) =  n+2$
		\end{enumerate}
		
		\item $ m\geq 5$ and $n\leq \lceil\frac{m}{2}\rceil-2$.\\
		Based on Theorem \ref{rvccycle}, we have $rvcl(C_m\diamond K_n)\geq \lceil\frac{m}{2}\rceil$. 
		Next, we establish $rvcl(C_m\diamond K_n)\leq \lceil\frac{m}{2}\rceil$ by constructing the following vertex coloring. $c:(V(C_m\diamond K_n)) \longrightarrow[1, \lceil\frac{m}{2}\rceil]$.
		\begin{enumerate}
			\item $n$ odd.
			\begin{center}
				$c(u_i)=	((i-1)~mod~ (\lceil\frac{m}{2}\rceil))+1$ for $i\in[1,m]$.
			\end{center}
			\begin{center}
				$\begin{array}{ccl}
					c(u_{i,j}^k)&=&\left\{
					\begin{array}{ll}
						((i+(k-2))~mod~ (\lfloor\frac{m}{2}\rfloor))+1, & \hbox{for $i\in [1, \lfloor\frac{m}{2}\rfloor -1],$} \\
						~ & \hbox{$j=i+1$, $k\in[1,n]$;}\\
						(i+(k-2))~mod~ (\lfloor\frac{m}{2}\rfloor +1)+1, & \hbox{for $i\in [\lfloor\frac{m}{2}\rfloor,m]$, $k\in[1,n]$,}\\
						~ & \hbox{$i+(k-1)\leq m$, $j=i+1$ }\\
						~& \hbox{for $i\neq m$, } \\
						~& \hbox{$j=1$ for $i=m$;}\\
						((i+k-2)~mod~ (m))+1, & \hbox{for $i\in [\lfloor\frac{m}{2}\rfloor +1,m]$,}\\
						~& \hbox{$k\in[1,n]$, $(i+k-2)\geq m$, }\\
						~& \hbox{$j=i+1$ for $i\neq m$, }\\
						~& \hbox{$j=1$ for $i=m$.}
						
					\end{array}
					\right.
				\end{array}$
			\end{center}
			\item $n$ even.
			\begin{center}
				$c(u_i)=	((i-1)~mod~ \frac{m}{2})+1$ for $i\in[1,m]$.\\
				
			\end{center}
			\begin{center}
				$\begin{array}{ccl}
					c(u_{i,j}^k)&=&\left\{
					\begin{array}{ll}
						((i+(k-2))~mod~ (\frac{m}{2}))+1, & \hbox{for $i\in [1, \frac{m}{2}]$,} \\
						~ & \hbox{$j=i+1$, $k\in[1,n]$;}\\
						((i+k-1)~mod~(\frac{m}{2}))+1, & \hbox{for $i\in [\frac{m}{2}+1, m]$, $j=i+1$ } \\
						~& \hbox{for $i\in[\frac{m}{2}+1,m-1]$, }\\
						~& \hbox{$k\in[1,n-1]$, $(i+k-1)< m$;}\\
						((i+k)~mod~ (\frac{m}{2}))+1, & \hbox{for $i\in [\frac{m}{2}+1,m]$, $j=i+1$}\\
						~& \hbox{for $i\neq m$, }\\
						~& \hbox{$j=1$ for $i=m$, }\\
						~& \hbox{$k=n$, $(i+k)< m$;}\\
						((i+k-1))~mod~ (m))+1, & \hbox{for $i\in [\frac{m}{2}+1,m]$, $j=i+1$ }	\\
						~& \hbox{for $i\neq m$,}\\
						~& \hbox{ $j=1$ for $i=m$,}\\
						~& \hbox{$k\in[1,n-1]$, $(i+k-1)\geq m$;}\\
						((i+k)~mod~ (m)+1, & \hbox{for $i\in [\frac{m}{2}+1,m]$, $j=i+1$}\\
						~& \hbox{for $i\neq m$, }\\
						~& \hbox{$j=1$ for $i=m$,}\\
						~& \hbox{$k=n$, $(i+k)\geq m$.}
					\end{array}
					\right.
				\end{array}$
			\end{center}
		\end{enumerate}
		
		From these coloring rules, it is obtained that any two vertices in the core are always given different colors at a maximum distance of $\lfloor\frac{m}{2}\rfloor$. This ensures that for every two vertices in $C_m \diamond K_n$, there exists an RVP.
		
		Next, consider some conditions regarding the rainbow codes of the vertices in $C_m\diamond K_n$ after applying the coloring.
		\begin{enumerate}
			\item For even $n$, each vertex in a flare has a set of colors where not all elements are the same.
			\item For odd $n$, at most two flares have a set of colors where all elements are the same, but they have different distances to the color $R_{\lceil\frac{m}{2}\rceil}$.
			\item One color is used at most twice in the core.
			\item 
			For even $n$, when $c(u_i)=c(u_j)$ for distinct $i,j\in [1,m]$, according to (a), these two points must be adjacent to vertices with different sets of colors.
			
			For odd $n$, when $c(u_i)=c(u_j)$ for distinct $i,j\in [1,m]$, we obtain $d(u_i,R_{\lceil\frac{m}{2}\rceil})\neq d(u_j, R_{\lceil\frac{m}{2}\rceil}).$
			
			\item Each vertex within the core is adjacent to vertices in two flares, each colored with different sets of $n$ distinct colors. Since each flare is colored with a set of colors where all elements are distinct, every vertex in the core is at a distance of $1$ from vertices colored with $n$ distinct colors. Additionally, it is known that each vertex in a flare is at a distance of $1$ from vertices colored with $n-1$ distinct colors.
		\end{enumerate}
		
		Based on the above conditions, it is concluded that every vertex generates a unique rainbow code. Thus, $rvcl(C_m\diamond K_n)=\lceil\frac{m}{2}\rceil$.

	\end{enumerate}
\end{proof}
\subsection{Locating Rainbow Connection Number of Edge Corona of Complete Graphs with Complete Graphs}
In this subsection, we determine $rvc$ and $rvcl$ of the edge corona of complete graphs with order $m$ and $n$, denoted by $K_m\diamond K_n$. Previously, in graphs $P_m\diamond K_n$ and $C_m\diamond K_n$, the smallest value between their $rvcl$ was the $rvc$ of both graphs. Conversely, in this subsection, we demonstrate that $rvc(K_m\diamond K_n)$ is not the strict lower bound of $rvcl(K_m\diamond K_n)$.
Let $V(K_m\diamond K_n)=\{u_i|i\in[1,m]\} \cup \{u_{i,i+1}^k| i\in [1,m-1], k\in[1,n]\} \cup \{u_{m,1}^k|k\in[1,n]\}$. $E(K_m\diamond K_n)=\{u_iu_j|i,j\in[1,m-1], j=i+1\} \cup \{u_mu_1\}$ $\cup \{u_iu_{i,j}^k|i,j\in[1,m], k\in[1,n]\} \cup \{u_iu_{i-1,i}^k|i\in[2,m], k\in[1,n]\} \cup \{u_1u_{m-1,1}^k|k\in[1,n]\} \cup \{u_{i,i+1}^ku_{i,i+1}^l|i,j\in [1,m-1], k,l\in[1,n], k\neq l\} \cup \{u_{m,1}^ku_{m,1}^l k,l\in[1,n], k\neq l\}$. Figure \ref{kmkn} illustrates the edge corona of a complete graph of order $m$ with a complete graph of order $n$.
\begin{figure}
	\centering
	\caption{Graf $K_m \diamond K_n$.}
	\includegraphics[width=3.3in]{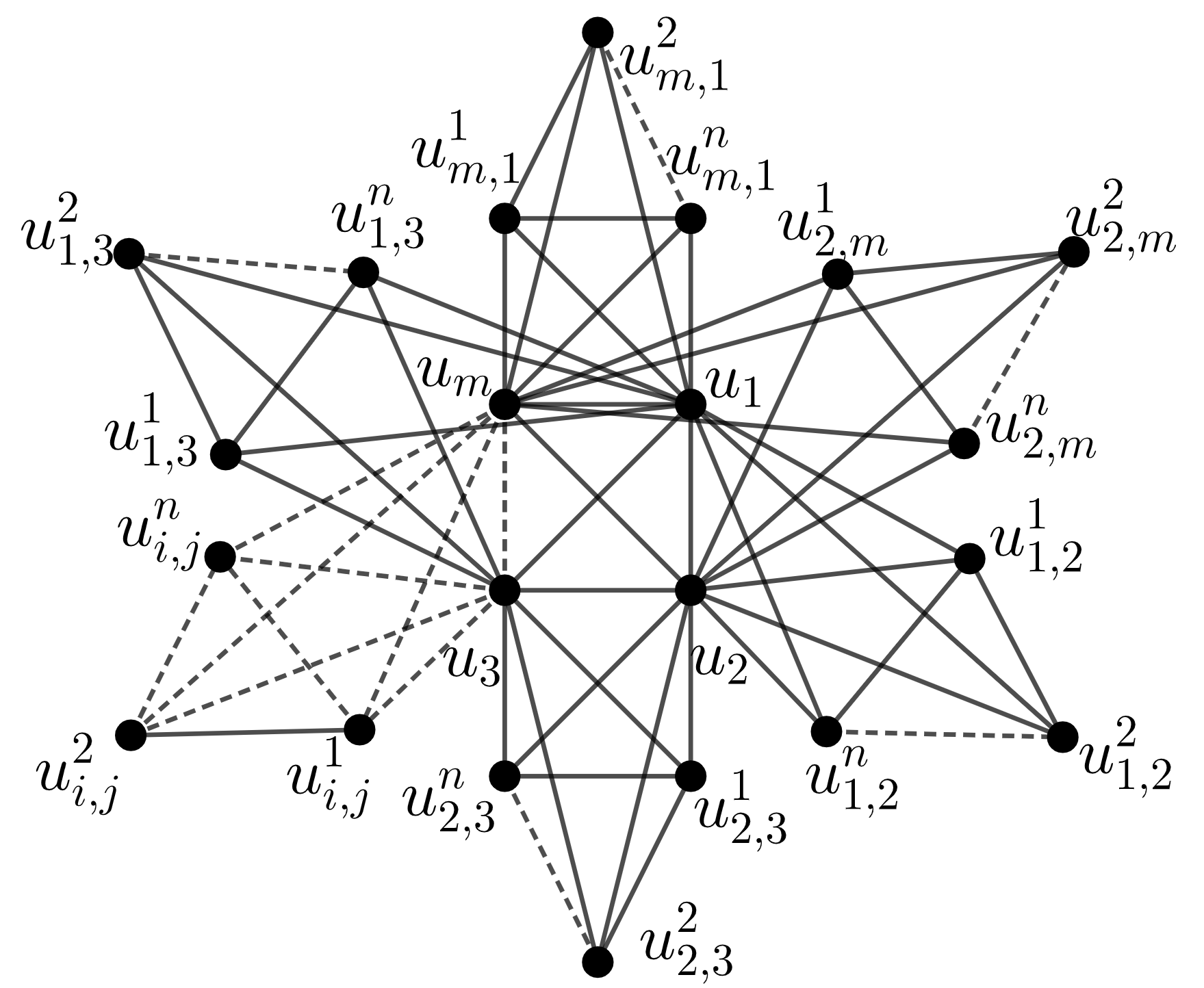}
	\label{kmkn}
\end{figure}
\begin{teorema}
	Let $m$ and $n$ be two integers where $m\geq3$ and $n\geq 2$. Let $K_m$ be a graph of order $m$, $H_n$ be a graph of order $n$. Then, 
	\begin{center}
		$rvc(K_m \diamond H_n) = \lceil\frac{m}{3}\rceil$.
	\end{center}
\end{teorema}
\begin{proof}
	Suppose $rvc(K_m \diamond H_n) = \lceil\frac{m}{3}\rceil-1$. This implies that there must be at least four distinct vertices $u,v,s,t$ in the core such that $c(u)=c(v)=c(s)=c(t)$. Consequently, there is no RVP connecting the vertices in the flare are adjacent to vertices $u,v$ with the vertices in the flare adjacent to vertices $s,t$, leading to a contradiction. Therefore, it must be the case that $rvc(K_m \diamond H_n) \geq \lceil\frac{m}{3}\rceil$. Next, we will demonstrate that $rvc(K_m \diamond H_n) \leq \lceil\frac{m}{3}\rceil$ by constructing a vertex coloring as follows. $c:(V(K_m \diamond H_n))\longrightarrow[1,\lceil\frac{m}{3}\rceil]$.
	
	\begin{center}
		$\begin{array}{ccl}
			c(u_i)&=&\left\{
			\begin{array}{ll}
				((i-1)~mod~\lceil\frac{m}{3}\rceil)+1& \hbox{for $i\leq 2(\lceil\frac{m}{3}\rceil)$,}\\
				i~mod~2(\lceil\frac{m}{3}\rceil)& \hbox{for $i >2(\lceil\frac{m}{3}\rceil)$.}
			\end{array}
			\right.
		\end{array}$
	\end{center}
	\begin{center}
		$c(u_{i,i+1}^k)=1$ for $i\in [1,m-1], k\in[1,n]$,\\
		$c(u_{m,1}^k)=1$ for $k\in[1,n]\}$.
	\end{center}
	
Each vertex within the core is adjacent to vertices in two flares, ensuring there is always an RVP connecting any two vertices in the core as well as in flare. Additionally, based on the coloring rule above, a single color is used at most three times in the core, ensuring there is always an RVP connecting any two vertices in the flare, and any two vertices between the core and the flare. Therefore, we have $rvc(K_m \diamond K_n)=\lceil\frac{m}{3}\rceil$.
\end{proof}

\begin{teorema}
	Let $m$ and $n$ be two integers where $m\geq3$, $n\geq 2$, and $m\leq n$. Let $K_m$ be a complete graph of order $m$, and $K_n$ be a complete graph of order $n$. Then, 
	
	\begin{center}
		$\begin{array}{ccl}
			rvcl(K_{m} \diamond K_n)&=&\left\{
			\begin{array}{ll}
				n+1, & \hbox{for $n\geq |E(K_m)|-1$ ;}\\
				n +2, & \hbox{$n <|E(K_m)|-1 $.}\\
			\end{array}
			\right.
		\end{array}$
	\end{center}
\end{teorema}

\begin{proof}
	The proof of the lower bound is divided into two cases as follows.
	\begin{enumerate}
		\item $n \geq |E(K_m)|-1$.\\
		Based on Lemma \ref{lemma1}, we have $rvcl(K_m\diamond K_n)\geq n+1$. 
		\item $n < |E(K_m)|-1$. \\
		Suppose $rvcl(K_m \diamond K_n)=n+1$. Since $n < |E(K_m)|-1$, there must be at least two flares whose vertices are colored with the same set of colors, let's call this set of colors $A$. According to Lemma \ref{lemma1}, $A$ contains $n$ distinct colors. Without a loss of generality, let these two flares be adjacent with the edges $uv$ and $xy$. Next, let's consider some possible colorings of these flare.
		\begin{enumerate}
			\item If $|A \cup \{c(u)\} \cup \{c(v)\}|=n+1$ and $|A \cup \{c(x)\} \cup \{c(y)\}|=n+1$, then based on Lemma \ref{batasbawah}, we get $rvcl(K_m\diamond K_n)\geq n+2$.
			\item 
			If $|A \cup \{c(u)\} \cup \{c(v)\}|=n+1$ and $|A \cup \{c(x)\} \cup \{c(y)\}|=n$, then $x$ must be colored with the same color as the colors used in the flare associated with vertices $x$ and $y$. Without a loss of generality, $c(x)=1$. Consequently, $x$ must be connected to a vertex colored with a color not used in the flare. However, this results in vertex $x$ having a representation of $1$ with respect to the colors $\{2, 3, ..., n+1\}$. Therefore, $x$ shares the same rainbow codes with one of the vertices in the flare adjacent with the edge $uv$, which leads to a contradiction. The same argument applies if $|A \cup \{c(u)\} \cup \{c(v)\}|=n$ and $|A \cup \{c(x)\} \cup \{c(y)\}|=n+1$.
			
			\item If $|A \cup \{c(u)\} \cup \{c(v)\}|=n$ and $|A \cup \{c(x)\} \cup \{c(y)\}|=n$, it means there are $n$ pairs of vertices that are colored the same and have entry $1$ with respect to the other $n-1$ color sets. Furthermore, without a loss of generality, assume that the color not included in $A \cup \{c(u)\} \cup \{c(v)\}$ is color $n+1$. In order for the rainbow code to be different, the $n$ pairs of vertices must have distinct entry with respect to $R_{n+1}$.
			
			Since $diam(K_m \diamond K_n)=3$, the possible entry that can be formed are only $2$ and $3$. Consequently, there are at least $n$ vertices that are at a distance of $3$ from $R_{n+1}$. In other words, there are at least $m-2$ other flares whose vertices are colored with the colors $\{1,2,...n\}$. Therefore, there are at least two vertices among these flares that have the same rainbow code, leading to a contradiction.
		\end{enumerate}
		Based on (a), (b), dan (c), we have $rvcl(K_m \diamond K_n) \geq n+2$.
	\end{enumerate}
	Next, we show $rvcl(K_m\diamond K_n)\leq n+1$ for $n\geq |E(K_m)|-1$ and $rvcl(K_m\diamond K_n)\leq n+2$ for $n< |E(K_m)|-1$ by providing the coloring steps on graph $K_m\diamond K_n$ as follows.
	\begin{enumerate}
		\item Assign color $i$ to $c(u_i)$ for $i \in[1,m]$.
		\item Let $A_p$ denote the set of $n$ distinct colors assigned to the flare indexed by $p$ for $p \in [1, |E(K_m)|]$.
		\item 
		After coloring all the vertices in the core, color each flare in such a way that the color set assigned to the vertices $u_{i,j}^k$ contains colors $i$ and $j$ for distinct $i,j\in [1,m]$, $k\in [1,n]$, and $A_p \neq A_q$ for distinct $p,q \in [1,|E(K_m)|]$.
		
	\end{enumerate} 
	
	Based on these coloring steps, it is obtained that every vertex in the core is colored differently, which results in every two non-adjacent vertices in graph $K_m \diamond K_n$ being connected by an RVP where all internal vertices are in the core. 
	
	Next, we show that every vertex generates a unique rainbow code by taking into account the following considerations.
	\begin{enumerate}
		\item All vertices in the core are colored with distinct colors, ensuring that each vertex in the core has a unique code.
		\item The vertices in each flare are assigned $n$ distinct colors, ensuring that every vertex generates a unique rainbow code.
		\item Each flare contains vertices assigned a set of colors where each color is distinct. Therefore, every pair of vertices between the flares has different rainbow codes.
		\item The vertices in the core are adjacent to at least $n$ vertices colored with distinct colors, whereas the vertices in the flares are only adjacent to $n-1$ vertices colored with distinct colors. Therefore, the rainbow code of a vertex in the core will never be the same as that of a vertex in a flare.
	\end{enumerate}
	Since every pair of vertices in $K_m \diamond K_n$ are connected by an RVP and every vertex generates a unique rainbow code, we get $rvcl(K_m \diamond K_n) = n+1$ for $n\geq |E(K_m)|-1$ dan $rvcl(K_m\diamond K_n)= n+2$ for $n< |E(K_m)|-1$.
\end{proof}
\begin{figure}[h]
	\centering
	\caption{A locating rainbow coloring of $K_3 \diamond K_4$}
	\includegraphics[width=3.6in]{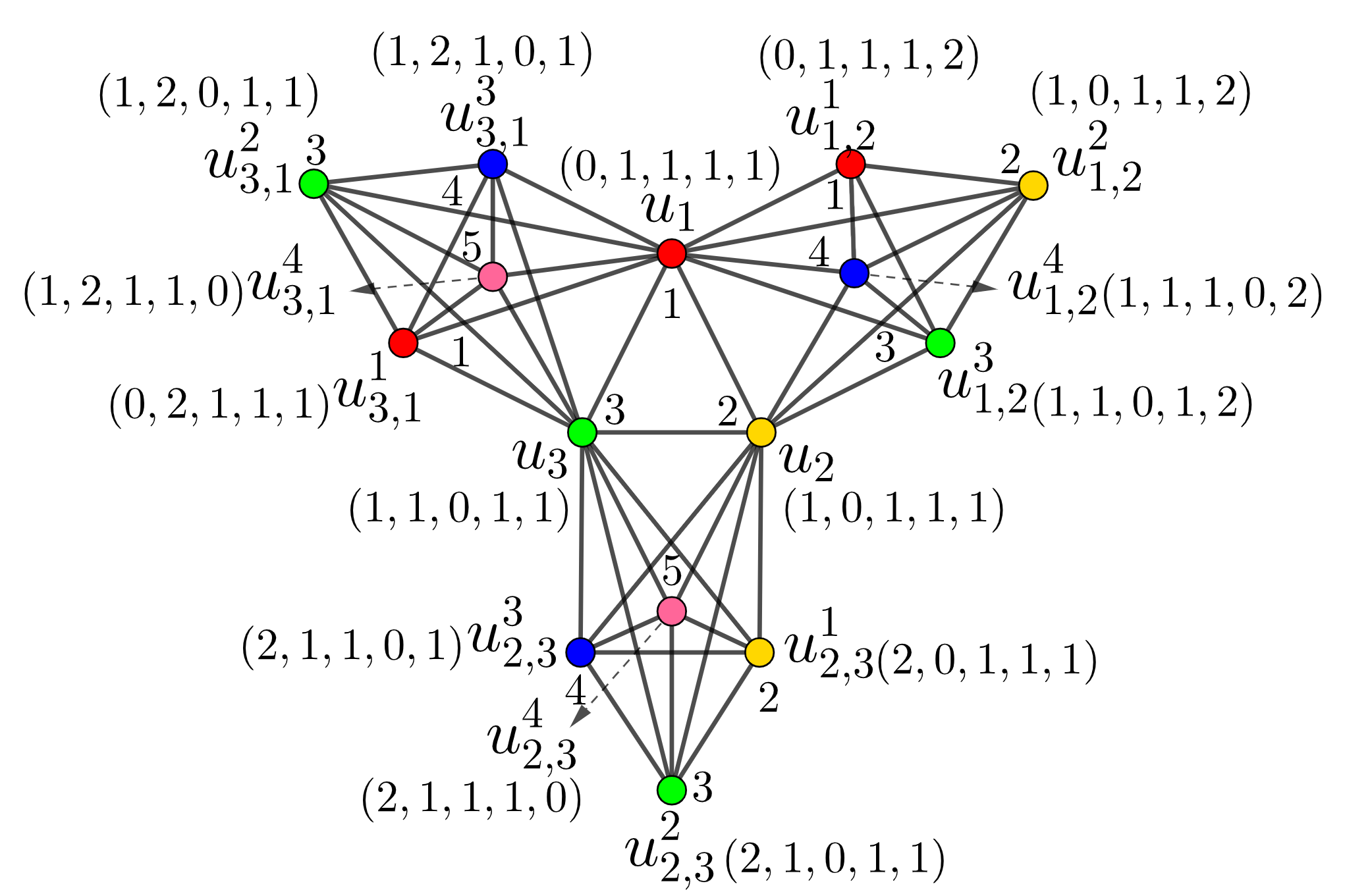}
	\label{fig5}
\end{figure}

Our research results indicate that the upper bounds are tight for any values of $m$ and $n$ where there exists a graph $T_m \diamond K_n$ with a $rvcl$ equal to the upper bound. As for the lower bound, it is satisfied by the graph $C_m \diamond K_n$ for $m=3$ or $m\geq 4$ and $n\geq m-1$, and $m\geq 5$ with $n\leq \lceil\frac{m}{2}\rceil-2$. Therefore, we conclude this study by posing an open problem: Is $max{\lceil\frac{m}{2}\rceil+1, n+2}$ the best lower bound for $rvcl(G_m\diamond K_n)$ for $m\geq 4$ and $\lceil\frac{m}{2}\rceil -1\leq n<m-1$?

\section*{Acknowledgements}

The authors would like to thank the Ministry of Education, Culture, Research, and Technology for supporting this research through the Doctoral Dissertation Research Grant program; LPDP from the Ministry of Finance Indonesia; and Institut Teknologi Bandung for their support in this research. 

\end{document}